\documentclass[12pt]{article}

\usepackage{amsmath,amssymb,amsbsy,amsfonts,amsthm,latexsym,
            amsopn,amstext,amsxtra,amscd,stmaryrd,color}

\usepackage[mathscr]{eucal}

\usepackage{subfig,graphics,epsfig,color}
\usepackage{float}

\begin{document}

\newtheorem{theorem}{Theorem}
\newtheorem{lemma}{Lemma}
\newtheorem{corollary}{Corollary}
\newtheorem*{computation}{Computation}
\newtheorem{proposition}{Proposition}

\newtheorem{conj}{Conjecture}

\theoremstyle{definition}
\newtheorem*{definition}{Definition}
\newtheorem*{remark}{Remark}
\newtheorem*{example}{Example}


\def\cA{\mathcal A}
\def\cB{\mathcal B}
\def\cC{\mathcal C}
\def\cD{\mathcal D}
\def\cE{\mathcal E}
\def\cF{\mathcal F}
\def\cG{\mathcal G}
\def\cH{\mathcal H}
\def\cI{\mathcal I}
\def\cJ{\mathcal J}
\def\cK{\mathcal K}
\def\cL{\mathcal L}
\def\cM{\mathcal M}
\def\cN{\mathcal N}
\def\cO{\mathcal O}
\def\cP{\mathcal P}
\def\cQ{\mathcal Q}
\def\cR{\mathcal R}
\def\cS{\mathcal S}
\def\cU{\mathcal U}
\def\cT{\mathcal T}
\def\cV{\mathcal V}
\def\cW{\mathcal W}
\def\cX{\mathcal X}
\def\cY{\mathcal Y}
\def\cZ{\mathcal Z}


\def\sA{\mathscr A}
\def\sB{\mathscr B}
\def\sC{\mathscr C}
\def\sD{\mathscr D}
\def\sE{\mathscr E}
\def\sF{\mathscr F}
\def\sG{\mathscr G}
\def\sH{\mathscr H}
\def\sI{\mathscr I}
\def\sJ{\mathscr J}
\def\sK{\mathscr K}
\def\sL{\mathscr L}
\def\sM{\mathscr M}
\def\sN{\mathscr N}
\def\sO{\mathscr O}
\def\sP{\mathscr P}
\def\sQ{\mathscr Q}
\def\sR{\mathscr R}
\def\sS{\mathscr S}
\def\sU{\mathscr U}
\def\sT{\mathscr T}
\def\sV{\mathscr V}
\def\sW{\mathscr W}
\def\sX{\mathscr X}
\def\sY{\mathscr Y}
\def\sZ{\mathscr Z}


\def\fA{\mathfrak A}
\def\fB{\mathfrak B}
\def\fC{\mathfrak C}
\def\fD{\mathfrak D}
\def\fE{\mathfrak E}
\def\fF{\mathfrak F}
\def\fG{\mathfrak G}
\def\fH{\mathfrak H}
\def\fI{\mathfrak I}
\def\fJ{\mathfrak J}
\def\fK{\mathfrak K}
\def\fL{\mathfrak L}
\def\fM{\mathfrak M}
\def\fN{\mathfrak N}
\def\fO{\mathfrak O}
\def\fP{\mathfrak P}
\def\fQ{\mathfrak Q}
\def\fR{\mathfrak R}
\def\fS{\mathfrak S}
\def\fU{\mathfrak U}
\def\fT{\mathfrak T}
\def\fV{\mathfrak V}
\def\fW{\mathfrak W}
\def\fX{\mathfrak X}
\def\fY{\mathfrak Y}
\def\fZ{\mathfrak Z}


\def\C{{\mathbb C}}
\def\F{{\mathbb F}}
\def\K{{\mathbb K}}
\def\L{{\mathbb L}}
\def\N{{\mathbb N}}
\def\Q{{\mathbb Q}}
\def\R{{\mathbb R}}
\def\Z{{\mathbb Z}}


\def\eps{\varepsilon}
\def\mand{\qquad\mbox{and}\qquad}
\def\\{\cr}
\def\({\left(}
\def\){\right)}
\def\[{\left[}
\def\]{\right]}
\def\<{\langle}
\def\>{\rangle}
\def\fl#1{\left\lfloor#1\right\rfloor}
\def\rf#1{\left\lceil#1\right\rceil}
\def\le{\leqslant}
\def\ge{\geqslant}
\def\ds{\displaystyle}

\def\xxx{\vskip5pt\hrule\vskip5pt}
\def\yyy{\vskip5pt\hrule\vskip2pt\hrule\vskip5pt}
\def\imhere{ \xxx\centerline{\sc I'm here}\xxx }

\newcommand{\commB}[1]{\marginpar{
\vskip-\baselineskip \raggedright\footnotesize
\itshape\hrule\smallskip
\begin{color}{blue} #1\end{color}
\par\smallskip\hrule}}

\newcommand{\commI}[1]{\marginpar{
\vskip-\baselineskip \raggedright\footnotesize
\itshape\hrule\smallskip
\begin{color}{red} #1\end{color}
\par\smallskip\hrule}}


\def\e{\mathbf{e}}
\def\Gnp{G_n(p)}
\def\Gdp{G_d(p)}
\def\Gnq{G_n(q)}
\def\pow{{1-1/n}}


\title{\bf On Gauss sums and the evaluation of Stechkin's constant}

\author{
{\sc William D.~Banks} \\
{Department of Mathematics} \\
{University of Missouri} \\
{Columbia, MO 65211 USA} \\
{\tt bankswd@missouri.edu} 
\and
{\sc Igor E.~Shparlinski} \\
{Department of Pure Mathematics} \\
{University of New South Wales} \\
{Sydney, NSW 2052, Australia} \\
{\tt igor.shparlinski@unsw.edu.au}}

\date{\empty}
\pagenumbering{arabic}

\maketitle

\newpage

\begin{abstract}
For the Gauss sums which are defined by
$$
S_n(a,q):=\sum_{x\bmod q}\e(ax^n/q),
$$
Stechkin~(1975) conjectured that the quantity
$$
A:=\sup_{n,q\ge 2}~\max_{\gcd(a,q)=1}\frac{\bigl|S_n(a,q)\bigr|}{q^{1-1/n}}
$$
is finite. Shparlinski~ (1991) proved that $A$ is finite, but in the
absence of effective bounds on the sums $S_n(a,q)$ the precise
determination of $A$ has remained intractable for many years.
Using recent work of Cochrane and Pinner~(2011) on
Gauss sums with prime moduli, in this paper we show that with the constant given by
$$
A=\bigl|S_6(\hat a,\hat q)\bigr|/\hat q^{1-1/6}=4.709236\ldots,
$$
where $\hat a:=4787$ and
$\hat q:=4606056=2^3{\cdot}3^2{\cdot}7{\cdot}13{\cdot}19{\cdot}37$,
one has the sharp inequality 
$$
\bigl|S_n(a,q)\bigr|\le A\,q^{1-1/n}
$$
for all $n,q\ge 2$ and all $a\in\Z$ with $\gcd(a,q)=1$.
One interesting aspect of our method is that we apply
effective lower bounds for the center density in the sphere
packing problem due to Cohn and Elkies~(2003)
to optimize the running time of our primary computational algorithm.
\end{abstract}

\section{Introduction}

In this paper we study the Gauss sums defined by
$$
S_n(a,q):=\sum_{x\bmod q}\e(ax^n/q)\qquad(n,q\ge 2,~a\in\Z)
$$
where $\e(t):=\exp(2\pi it)$ for all $t\in\R$.  Since
$S_n(a,q)=dS_n(a/d,q/d)$ for any integer $d\ge 1$ that
divides both $a$ and $q$, for given $n$ and $q$ it is natural to investigate
the quantity
$$
\Gnq := \max_{\gcd(a,q)= 1} \bigl|S_n(a,q)\bigr|,
$$
which is the largest absolute value of the ``irreducible'' Gauss sums
for a given modulus $q$ and exponent $n$.
It is well known (see Stechkin~\cite{Ste}) that for some constant $C(n)$ 
that depends only on $n$ one has a bound of the form
$$
\Gnq \le C(n)\,q^{1-1/n}\qquad(q\ge 2),
$$
and therefore the number
$$
A(n) := \sup_{q\ge 2}~\Gnq/q^{1-1/n}
$$
is well-defined and finite for each $n\ge 2$.
Stechkin~\cite{Ste} showed that the bound
\begin{equation}
\label{eq:Ste}
A(n)\le \exp\bigl(O(\log\log 3n)^2\bigr)\qquad (n\ge 2),
\end{equation}
holds, and he conjectured that 
for some absolute constant $C$ one has
\begin{equation}
\label{eq:ConjSte}
A(n)\le C\qquad(n\ge 2).
\end{equation}
Shparlinski~\cite{Shp} proved Stechkin's conjecture in the stronger form
\begin{equation}
\label{eq:Shp}
A(n) = 1 + O(n^{-1/4 + \varepsilon})\qquad(n\ge 2).
\end{equation}
We remark that the estimate~\eqref{eq:Shp} has been subsequently strengthened
by Konyagin and Shparlinski
(see~\cite[Theorem~6.7]{KoSh}) to
\begin{equation}
\label{eq:KoSh}
A(n) = 1 + O\(n^{-1} \tau(n) \log n\)\qquad(n\ge 2), 
\end{equation}
where $\tau(\cdot)$ is the divisor function.
In the opposite direction, it has been shown in~\cite[Theorem~6.7]{KoSh}
that for infinitely many integers $n$ one has the lower bound
\begin{equation}
\label{eq:KoSh LB}
A(n) > 1+ n^{-1} \exp\(\frac{0.43\,\log n}{\log \log  n}\).
\end{equation}
We also note that the lower bound 
\begin{equation}
\label{eq:A1}
A(n)\ge 1
\end{equation}
holds for all $n\ge 2$ as one sees by applying \cite[Lemma~6.4]{KoSh}
with $m:=n$ and $q:=p$ for some prime $p\nmid n$.

The validity of~\eqref{eq:Shp}
leads naturally to the problem of determining the
exact value of \emph{Stechkin's constant}
$$
A:=\max_{n\ge 2} A(n),
$$
and it is this problem that is the focus of the present paper.

\begin{theorem}
\label{thm:main} 
We have $A(n)<A(6)$ for all $n\ge 2$, $n\ne 6$.  In particular,
with the constant
$$
A:=A(6)=\frac{\bigl|S_6(4787,4606056)\bigr|}{4606056^{\,5/6}}=4.70923685314526794358\ldots
$$
one has
$$
\bigl|S_n(a,q)\bigr|\le A\,q^{1-1/n}
$$
for all $n,q\ge 2$ and $a\in\Z$ with $\gcd(a,q)=1$.
\end{theorem} 

The results described above have all been obtained by reducing
bounds for the general sums $\Gnq$ to bounds on sums $\Gnp$
with a prime modulus.  There are several different (and elementary) ways 
to show that the bound $\Gnp\le np^{1/2}$
holds (see, e.g., Lidl and Niederreiter~\cite[Theorem~5.32]{LN}), 
a result that plays the key role in Stechkin's proof of~\eqref{eq:Ste} in~\cite{Ste}.
Stechkin~\cite{Ste} observed that in order to prove the  conjecture~\eqref{eq:ConjSte}
one simply needs a bound on $\Gnp$ which remains nontrivial for all $n\le p^{\vartheta}$
with some fixed $\vartheta > 1/2$. 
The first bound of this type, valid for any fixed $\vartheta<4/7$,
is given in~\cite{Shp}; taken together with the argument of Stechkin~\cite{Ste} 
this
leads to~\eqref{eq:Shp}. An improvement by Heath-Brown and Konyagin (see~\cite[Theorem~1]{HBK})
of the principal result of~\cite{Shp}, along with some additional arguments,
leads to the stronger estimate~\eqref{eq:KoSh}; see~\cite[Chapter~6]{KoSh} for details.

The problem of determining $A$ explicitly involves much more effort than
that of simply performing a single direct computation.  The starting point in our
proof of Theorem~\ref{thm:main} is the replacement of the bound of~\cite[Theorem~1]{HBK}
with a more recent \emph{effective} bound on Gauss sums due to Cochrane and Pinner~\cite{CochPin};
see~\eqref{eq:B(d,p)-defn} below.  Using this bound one sees that each number
$A(n)$ can be computed in a finite number of steps.  However, the number of
steps required is quite huge even for small values of $n$, and the direct
computation of $A(n)$ is therefore exceedingly slow (especially when $n$ is prime).  It is infeasible to compute $A(n)$
over the entire range of values of $n$ that are needed to yield
the proof of Theorem~\ref{thm:main} directly from the bound
of~\cite{CochPin} 
combined with the argument of~\cite{Shp}. Instead,
to obtain Theorem~\ref{thm:main} we establish the upper bound
$A(n)<4.7$ for all $n>6$ using a combination
of previously known bounds and some new bounds.

Our underlying approach has been to modify and extend 
the techniques of~\cite[Chapter~6]{KoSh}
to obtain an effective version of~\cite[Theorem~6.7]{KoSh}. More precisely,
in Propositions~\ref{prop:technical} and~\ref{prop:technical2} we give general conditions
under which one can disregard the value of $\Gnp$ in computation of $A(n)$.  Special
cases of these results, stated as Corollaries~\ref{cor:scissors}--\ref{cor:bear}, have been used 
to perform the main computation described at the beginning of \S\ref{sec:numerical}.
An interesting aspect of our method is that Corollary~\ref{cor:bear}, which shows
that $\Gnp$ can be disregarded if $n\ge 2000$, $p\ge 8.5\times 10^6$, and
$(p-1)/\gcd(n,p-1)\ge 173$, essentially relies on effective lower bounds for the center
density in the sphere packing problem due to Cohn and Elkies~\cite{CohnElkies}.  In the
absence of these lower bounds, the running time of our primary computational algorithm
would have increased by a factor of at least one thousand.  We also remark that the
criteria presented in Corollaries~\ref{cor:scissors}--\ref{cor:cord} allow for
early termination of the program as the sums over $x$ in~\eqref{eq:bracketboundx}, 
\eqref{eq:bracketboundy} and~\eqref{eq:bracketboundz}
are monotonically increasing and avoid the use of complex numbers.

\section{Proof of Theorem~\ref{thm:main}}

\subsection{Theoretical results}
\label{sec:theoretical}

In what follows, the letter $n$ always denotes a natural number,
and the letter $p$ always denotes a prime number.  

We recall that $\Gnp=\Gdp$ holds whenever
$\gcd(n,p-1)=d$; see\break \cite[Lemma~6.6]{KoSh}.
Our main technical tool for proving Theorem~\ref{thm:main}
is the bound
\begin{equation}
\label{eq:B(d,p)-defn}
\Gdp\le B(d,p):=\min\{(d-1)p^{1/2},\lambda d^{5/8}p^{5/8},\lambda d^{3/8}p^{3/4}\}+1,
\end{equation}
where
$$
\lambda:=2\cdot 3^{-1/4}=1.519671\ldots;
$$
this is the main result of Cochrane and Pinner~\cite[Theorem~1.2]{CochPin}.

For a given prime $p$ and natural number $n$, let $v_p(n)$ denote the
greatest integer $m$ for which $p^m\mid n$ (that is, $v_p(\cdot)$ is the
usual $p$-adic valuation). Arguing as in~\cite[Chapter~6]{KoSh} we have
\begin{equation}
\label{eq:A(n)calc}
A(n)=A_1(n)A_2(n),
\end{equation}
where
\begin{align*}
A_1(n)&:=\prod_{p\,\mid\,n}\max_{1\le m\le v_p(n)+2}
\big\{\Gnp/p^{m(1-1/n)},1\big\},\\
A_2(n)&:=\prod_{d\,\mid\, n}
\prod_{\substack{p\,\nmid\,n\\ \gcd(n,p-1)=d\\ B(d,p)>p^{1-1/n}}}
\max\big\{\Gdp/p^{1-1/n},1\big\}.
\end{align*}
Note that for fixed $d$ and $n$ there are only finitely many primes $p$
for which $B(d,p)>p^{1-1/n}$.  For our purposes below, we recall that the bound
\begin{equation}
\label{eq:A(n)bd}
A(n)\le n^{3/n}A_2(n)
\end{equation}
holds; see~\cite[p.~42]{KoSh}.

\begin{lemma}
\label{lem:expsumest}
Let $b_1,\ldots,b_m$ be real numbers with $|b_j|<p/2$ for each $j$,
and suppose that
$$
\sum_{j=1}^m b_j^2\ge C.
$$
Then
\begin{equation}
\label{eq:realsumbd1}
\Re\sum_{j=1}^m\e(b_j/p)\le m-\frac{8C}{p^2},
\end{equation}
where $\Re\,z$ denotes the real part of $z\in\C$.
Moreover, if $|b_j|<p/4$ for each $j$, then
\begin{equation}
\label{eq:realsumbd2}
\Re\sum_{j=1}^m\e(b_j/p)\le m-\frac{16C}{p^2}.
\end{equation}
\end{lemma}

\begin{proof}
The first bound~\eqref{eq:realsumbd1} is~\cite[Lemma~4.1]{KoSh};
the proof is based on the inequality $\cos(2\pi u)\le 1-8u^2$
for $u\in[-\tfrac12,\tfrac12]$. The second bound~\eqref{eq:realsumbd2}
is proved similarly using the inequality $\cos(2\pi u)\le 1-16u^2$
for $u\in[-\tfrac14,\tfrac14]$.
\end{proof}

To state the next result, we introduce some notation.  As usual,
we denote by $\varphi(\cdot)$ the Euler function.  In what follows,
for a fixed odd prime $p$ and any $b\in\Z$ we denote by
$\llbracket b\rrbracket_p$ the unique integer such that
$b\equiv \llbracket b\rrbracket_p\pmod p$ and
$-p/2<\llbracket b\rrbracket_p<p/2$.  We also
denote by $g$ a fixed generator of the multiplicative group $\F_p^*$
of the finite field $\F_p:=\Z/p\Z$.

For any $r\ge 2$ let
$$
C_r:=\(r\,\gamma_{r-1}^{r-1}\)^{-1/r}\mand
K_r:=4(1-1/f_r)C_r
$$
where $\gamma_k$ denotes the $k$-th Hermite constant
(see Conway and Sloane~\cite{ConSlo}),
and $f_r$ is the least natural number such that $\varphi(f_r)\ge r$.

Finally,
for fixed $n$ and $p$ we put $d:=\gcd(n,p-1)$ and $t:=(p-1)/d$.

\begin{proposition}
\label{prop:technical}
Fix $n$, $p$ and $\Theta\ge 1$. Suppose that $\varphi(t)\ge r\ge 2$, that
the inequalities
\begin{equation}
\label{eq:bracketbound}
\sum_{x=1}^{t} \llbracket g^{dx+y}\rrbracket_p^2
\ge \Theta F_r(t,p)
\qquad (1\le y\le d)
\end{equation}
hold with
$$
F_r(t,p):=\biggl(\frac{p^{2(r-1)}t}{r\, \gamma_{r-1}^{r-1}}\biggr)^{1/r}=C_rp^{2-2/r}t^{1/r},
$$
and that the inequality
\begin{equation}
\label{eq:pbound}
p^{1-1/n}\ge 1+\lambda\biggl(\frac{p^{5}\log p}{\Theta K_r(p-1)^{1/r}}\biggr)^{3r/(16r-8)}
\end{equation}
holds. Then $\Gdp\le p^{1-1/n}$.
\end{proposition}

\begin{proof}
We can assume that $B(d,p)>p^{1-1/n}$, for otherwise the result
follows immediately from~\eqref{eq:B(d,p)-defn}.

Write
$$
\Gdp=1+d\max_{1\le y\le d}\biggl|\,\sum_{x=1}^t\e(g^{dx+y}/p)\biggr|.
$$
Replacing $F_r(t,p)$ with $\Theta F_r(t,p)$ in the proof of~\cite[Theorem~4.2]{KoSh},
taking into account~\eqref{eq:bracketbound} and our hypothesis that
$\varphi(t)\ge r\ge 2$, we see that
\begin{equation}
\label{eq:bonkers}
\Gdp\le 1+d\bigl(t-4\Theta F_r(t,p)(1-1/t)p^{-2}\bigr).
\end{equation}
Since $t\ge f_r$ it follows that
$$
\Gdp\le 1+d\bigl(t-4\Theta C_rp^{-2/r}t^{1/r}(1-1/t)\bigr)
\le p-\Theta K_rdp^{-2/r}t^{1/r},
$$
and recalling that $t:=(p-1)/d$ this leads to the bound
\begin{equation}
\label{eq:vase}
\Gdp\le p-\Theta K_rd^{1-1/r}p^{-2/r}(p-1)^{1/r}.
\end{equation}

On the other hand, combining~\eqref{eq:B(d,p)-defn} and~\eqref{eq:pbound} we have
$$
1+\lambda d^{3/8}p^{3/4}\ge B(d,p)>p^{1-1/n}\ge
1+\lambda\biggl(\frac{p^{5}\log p}{\Theta K_r(p-1)^{1/r}}\biggr)^{3r/(16r-8)},
$$
which in turn yields the inequality
$$
\Theta K_rd^{24/25}p^{-2/25}(p-1)^{1/25}\ge d^{-1}p\log p.
$$
In view of~\eqref{eq:vase} we deduce that
$$
\Gdp\le p-d^{-1}p\log p\le p^{1-1/d}\le p^{1-1/n}
$$
as required.
\end{proof}

\begin{corollary}
\label{cor:scissors}
Suppose that $n\ge 2000$, $\varphi(t)\ge 25$, $p\ge 375000$, and
\begin{equation}
\label{eq:bracketboundx}
\sum_{x=1}^{t} \llbracket g^{dx+y}\rrbracket_p^2
\ge p^{48/25}t^{1/25} \qquad (1\le y\le d).
\end{equation}
Then $\Gdp\le p^{1-1/n}$.
\end{corollary}

\begin{proof}
Taking $r:=25$ in the statement of Proposition~\ref{prop:technical},
we observe that
$$
\gamma_{24}=4,\qquad
C_{25}=(5\cdot 2^{24})^{-2/25},\qquad
K_{25}=\tfrac{112}{29}\,C_{25}.
$$
We put $\Theta:=C_{25}^{-1}$ so that~\eqref{eq:bracketbound} and~\eqref{eq:bracketboundx}
are equivalent, and then we verify that the inequality~\eqref{eq:pbound} holds under
the conditions of the corollary.
\end{proof}

Similarly, with the choice $\Theta:=2C_{25}^{-1}$ we obtain the following statement.

\begin{corollary}
\label{cor:string}
Suppose that $n\ge 2000$, $\varphi(t)\ge 25$, $p\ge 6500$, and
\begin{equation}
\label{eq:bracketboundy}
\sum_{x=1}^{t} \llbracket g^{dx+y}\rrbracket_p^2
\ge 2\,p^{48/25}t^{1/25} \qquad (1\le y\le d).
\end{equation}
Then $\Gdp\le p^{1-1/n}$.
\end{corollary}

\begin{corollary}
\label{cor:cord}
Suppose that $n\ge 2000$, $\varphi(t)\ge 10$, $p\ge 8000$, and
\begin{equation}
\label{eq:bracketboundz}
\sum_{x=1}^{t} \llbracket g^{dx+y}\rrbracket_p^2
\ge 13\,p^{16/9}t^{1/9} \qquad (1\le y\le d).
\end{equation}
Then $\Gdp\le p^{1-1/n}$.
\end{corollary}

\begin{proof}
Taking $r:=9$ in the statement of Proposition~\ref{prop:technical},
we observe that
$$
\gamma_{8}=2,\qquad
C_{9}=48^{-2/9},\qquad
K_{9}=\tfrac{40}{11}\,C_{9}.
$$
We put $\Theta:=13\,C_{9}^{-1}$ so that~\eqref{eq:bracketbound} and~\eqref{eq:bracketboundz}
are equivalent, and then we verify that the inequality
\eqref{eq:pbound} holds under the conditions of the corollary.
\end{proof}

When $\Theta:=1$ the bound~\eqref{eq:bracketbound} holds for
$\varphi(t)\ge r\ge 2$ as is demonstrated in the proof
of~\cite[Lemma~4.2]{KoSh}.  Moreover, in this case we have the
following variant of Proposition~\ref{prop:technical}.

\begin{proposition}
\label{prop:technical2}
Fix $n$ and $p$. Suppose that $\varphi(t)\ge r\ge 2$,
and that the inequalities
\begin{equation}
\label{eq:tbound2}
p^{2/r}t^{1-1/r}>32rC_r
\end{equation}
and
\begin{equation}
\label{eq:pbound2}
p^{1-1/n}\ge 1+\lambda
\biggl(\frac{p^{5}\log p}{2K_r(p-1)^{1/r}}\biggr)^{3r/(16r-8)}
\end{equation}
hold. Then $\Gdp\le p^{1-1/n}$.
\end{proposition}

\begin{proof}
Fix $y$ in the range $1\le y\le d$.
For every set $\cI$ containing precisely $r$ consecutive integers,
the proof of~\cite[Theorem~4.2]{KoSh} shows that
$$
\sum_{x\in\cI}\llbracket g^{dx+y}\rrbracket_p^2\ge \frac{rF_r(t,p)}{t}.
$$
where $F_r(t,p)$ is as in Proposition~\ref{prop:technical}. 
If it is the case that
$$
\sum_{x\in\cI}\llbracket g^{dx+y}\rrbracket_p^2\ge \frac{2rF_r(t,p)}{t},
$$
then Lemma~\ref{lem:expsumest} gives
\begin{equation}
\label{eq:intIest}
\Re\sum_{y\in\cI}\e(b_y/p)\le r-\frac{16rF_r(t,p)}{tp^2}.
\end{equation}
On the other hand, suppose that
$$
\sum_{x\in\cI}\llbracket g^{dx+y}\rrbracket_p^2<\frac{2rF_r(t,p)}{t}.
$$
From~\eqref{eq:tbound2} it follows that
$$
\sum_{x\in\cI}\llbracket g^{dx+y}\rrbracket_p^2<\frac{p^2}{16},
$$
hence $\bigl|\llbracket g^{dx+y}\rrbracket_p\bigr|<p/4$ for each $x\in\cI$.
By Lemma~\ref{lem:expsumest} we again obtain~\eqref{eq:intIest}.

Since~\eqref{eq:intIest} holds for every set $\cI$ of $r$ consecutive integers,
it follows that
$$
\Re\sum_{y=1}^{t-1}\e(b_y/p)\le t-\frac{16F_r(t,p)}{p^2}.
$$
Writing
$$
\Gdp=1+d\max_{1\le y\le d}\biggl|\sum_{x=1}^t\e(g^{dx+y}/p)\biggr|
$$
and proceeding as in the proof of~\cite[Theorem~4.2]{KoSh} we see that
\begin{equation}
\label{eq:bonkers2}
\Gdp\le 1+d\bigl(t-8 F_r(t,p)(1-1/t)p^{-2}\bigr).
\end{equation}
We complete the proof of Proposition~\ref{prop:technical2}
by following that of Proposition~\ref{prop:technical},
taking $\Theta:=1$ and applying~\eqref{eq:bonkers2} instead of~\eqref{eq:bonkers}.
\end{proof}

\begin{corollary}
\label{cor:bear}
Suppose that $n\ge 2000$, $p\ge 8.5\times 10^6$, and $t\ge 173$.
Then $\Gdp\le p^{1-1/n}$.
\end{corollary}

\begin{proof}
We put $r:=30$ in the statement of Proposition~\ref{prop:technical2}.
For any $t\ge 173$, we have $\varphi(t)\ge r$, and the
inequalities~\eqref{eq:tbound2} and~\eqref{eq:pbound2} are readily
verified by taking into account that $2.08174<\gamma_{29}<3.90553$
(the lower bound on $\gamma_{29}$ follows from
$$
\gamma_k\ge\frac{1}{\pi}\,\bigl(2\zeta(k)\Gamma(1+k/2)\bigr)^{2/k},
$$
which was first stated by Minkowski and proved by Hlawka~\cite{Hlawka};
the upper bound on $\gamma_{29}$ follows from
Cohn and Elkies~\cite[Table~3]{CohnElkies}).
\end{proof}

\subsection{Numerical methods}
\label{sec:numerical}

\begin{computation}
\label{comp:fox}
For all $n\ge 2000$, $p\le 8.5\times 10^6$, and $t\ge 173$, the
inequality $\Gdp\le p^{1-1/n}$ holds.
\end{computation}

\begin{proof}[Description.]
For all $t\ge 637$ one sees that $B(d,p)\le p^{1-1/2000}$ for all primes
$p\le 8.5\times 10^6$; hence $\Gdp\le p^{1-1/n}$ holds in this case.

For $375000\le p\le 8.5\times 10^6$ and $173\le t\le 636$
we apply Corollary~\ref{cor:scissors}. Since the inequality
$\varphi(t)\ge 25$ is easily satisfied,
it suffices to verify that~\eqref{eq:bracketboundx} holds for all
such $p$ and $t$, which we have done.

Similarly, for $6500\le p\le 375000$ and $173\le t\le 636$
we apply Corollary~\ref{cor:string}, checking that
\eqref{eq:bracketboundy} holds for all such $p$ and $t$.

For the remaining primes $p\le 6500$ we have verified 
on a case-by-case basis that $\Gdp\le p^{1-1/2000}$ holds
whenever $t\ge 173$.
\end{proof}

Taking into account Corollary~\ref{cor:bear} and the above computation
along with the trivial bounds $G_1(p)=0$, $G_2(p)=p^{1/2}$
and $G_d(p)\le p$ when $d\ge 3$,
for every $n\ge 2000$ we deduce from~\eqref{eq:A(n)bd} that
\begin{equation}
\label{eq:coffee}
A(n)\le n^{3/n}\prod_{\substack{d\,\mid\, n\\ d\ge 3}}~
\prod_{\substack{p\equiv 1\,(\textrm{mod}~d)\\ (p-1)/d\le 172}}p^{1/n}
=n^{3/n}\mathop{\prod_{\substack{d\,\mid\, n\\ d\ge 3}}
~\prod_{t\le 172}}\limits_{dt+1\text{~is prime}}(dt+1)^{1/n}.
\end{equation}
This yields a useful but somewhat less precise bound
\begin{equation}
\label{eq:coffee2}
A(n)\le n^{3/n}(172n+1)^{172\tau(n)/n}.
\end{equation}
Combining~\eqref{eq:coffee2} with the explicit
bound of Nicolas and Robin~\cite{NicRob}
$$
\frac{\log\tau(n)}{\log 2}\le 1.54\,\frac{\log n}{\log\log n}
\qquad (n\ge 3),
$$
one sees that $A(n)<4.7$ for all $n\ge 456000$.

For smaller values of $n$, we have used the bound~\eqref{eq:coffee}
to check that the inequality $A(n)<A(6)$ holds for all $n$ in the range
$2000<n<456000$ apart from $677$ ``exceptional'' numbers,
which we collect together into a set
$$
\cE:=\{2002,2004,2010,\ldots,25200,27720,30240\}.
$$
We take $\cD$ to be the set of integers $d\ge 3$ such that either
$d\le 2000$ or else $d$ divides some number $n \in \cE$; the set
$\cD$ has $2710$ elements.

For each $d\in\cD$ and prime $p$ satisfying the conditions
$p\equiv 1\pmod d$, $(p-1)/d\le 172$, and $B(d,p)>p^{1-1/d}$,
we have computed the value of $\Gdp$ numerically to high precision;
this has been done for precisely $85112$ pairs $(p,d)$
altogether, and of these, all but $3618$ pairs have been
subsequently eliminated as the condition $\Gdp\le p^{1-1/d}$ is met;
for the surviving pairs, the value $\Gdp$ has been retained.  Having
these values at our disposal, we have been able to accurately estimate
the quantity $A_2(n)$ for all $n\le 2000$ and for all $n\in\cE$.
In view of~\eqref{eq:A(n)bd} we have found that $A(n)<4.7$
for all $n>6$.

It is well known that $A(2)=\sqrt{2}$, and 
using~\eqref{eq:A(n)calc} we are able to determine
$A(n)$ precisely for $n=3,4,5,6$ (see Table~\ref{tab:AN} 
in~\S\ref{sec:furthercalcs}).
We find that $A(n)<4.7$ for $2\le n\le 5$, whereas $A(6)>4.7$,
and the proof of Theorem~\ref{thm:main} is complete.

\section{Further results and conjectures}
\label{sec:furthercalcs}

In Table~\ref{tab:AN}, we list numerical upper bounds for $A(n)$ in the
range $3\le n\le 40$; each bound agrees with the exact value of $A(n)$ to within $10^{-8}$.

\begin{table}[H]
  \centering
\begin{tabular}{|c|c||c||c|c|}
\hline
$n$ & $A(n)$ & & $n$ & $A(n)$ \\
\hline
3 & 3.92853006 & & 22 & 1.46567511 \\
4 & 4.26259099 & & 23 & 1.31902122 \\
5 & 2.59880326 & & 24 & 1.77609946 \\
6 & 4.70923686 & & 25 & 1.42781090 \\
7 & 2.11936480 & & 26 & 1.60401011 \\
8 & 2.21026135 & & 27 & 1.54156739 \\
9 & 2.28069995 & & 28 & 1.35754104 \\
10 & 3.25099720 & & 29 & 1.14455967 \\
11 & 1.53359821 & & 30 & 1.69652491 \\
12 & 2.65269611 & & 31 & 1.00000000 \\
13 & 1.39611207 & & 32 & 1.51129998 \\
14 & 1.56950385 & & 33 & 1.31715766 \\
15 & 1.44795316 & & 34 & 1.18744155 \\
16 & 1.78417788 & & 35 & 1.23094084 \\
17 & 1.15247718 & & 36 & 1.78968236 \\
18 & 2.53272793 & & 37 & 1.19086823 \\
19 & 1.00000000 & & 38 & 1.08865451 \\
20 & 1.94022813 & & 39 & 1.31104883 \\
21 & 1.60324184 & & 40 & 1.47364476 \\ 
\hline
\end{tabular}
\caption{Values $A(n)$ with $3\le n\le 40$}
\label{tab:AN}
\end{table}

We observe that $A(19)=A(31)=1$.  On the basis of this and other numerical data
gathered for this project, we make the following

\begin{conj} 
\label{conj: A=1} We have $A(n)=1$ for infinitely
many natural numbers $n$.
\end{conj}

On the other hand, the average value of $A(n)$ is not too close to one
in the following sense.

\begin{proposition}
\label{prop:misc}
Put
$$
E(N):=\sum_{n=2}^N \bigl(A(n)-1\bigr)\qquad(N\ge 2).
$$
Then $E(N)\ge (2+o(1))\log N$ as $N\to\infty$.
\end{proposition}

\begin{proof}
Let $p\ge 5$ be an odd prime, and set $n:=(p-1)/2$.
It is easy to see that $S_n(a,p)=1+(p-1)\cos(2\pi a/p)$ if $p\nmid a$,
hence 
$$
G_n(p)\ge 1+(p-1)\cos(2\pi/p)=p-2\pi^2 p^{-1}+O(p^{-2}).
$$
Using this bound together with the estimate
$$
p^{1/n}=\exp\(\frac{2\log p}{p-1}\)=1+\frac{2\log p}{p}+O\bigg(\frac{\log^2p}{p^2}\bigg)
$$
it follows that
$$
A(n)\ge\frac{G_n(p)}{p^{1-1/n}}\ge 1+\frac{2\log p}{p}+O\bigg(\frac{\log^2p}{p^2}\bigg);
$$
therefore
\begin{align*}
E(N)&\ge\sum_{\substack{2\le n\le N\\2n+1\text{~is prime}}}\bigl(A(n)-1\bigr)
=\sum_{5\le p\le 2N+1}\bigl(A((p-1)/2)-1\bigr)\\
&\ge\sum_{5\le p\le 2N+1} \bigg(\frac{2\log p}{p}+O\bigg(\frac{\log^2p}{p^2}\bigg)\bigg)
=(2+o(1))\log N,
\end{align*}
and the proposition is proved.
\end{proof}

Combining Proposition~\ref{prop:misc} with the upper bound
$E(N)\ll(\log N)^3$, which follows immediately from~\eqref{eq:KoSh},
we see that 
$$
\log E(N)\asymp\log\log N,
$$ 
and it seems reasonable
to make the following

\begin{conj}
\label{conj: A > 1} For some constant $c\in(1,3)$ we have
$$
E(N)=(\log N)^{c+o(1)}\qquad(N\to\infty).
$$
\end{conj}

Although we have only computed $E(N)$ precisely in the limited
range $2\le N\le 40$, for large $N$ the value $E(N)$ is
closely approximated by the\break quantity 
$$
E_2(N):=\sum_{n=2}^N\bigl(A_2(n)-1\bigr),
$$
which therefore provides a reasonably tight lower bound for $E(N)$.
Using the data we collected for the proof of Theorem~\ref{thm:main}
we have computed $E_2(N)$ in the wider range $2\le N\le 2000$.
In Figures~\ref{fig:1.7},\ref{fig:1.76},\ref{fig:1.8} below
we have plotted the values $E_2(N)/(\log N)^c$ in the same range
with the choices $c=1.74$, $1.762$ and $1.78$, respectively 
(note that the scales are different along the vertical axes).  
These 
data suggest that $(\log E_2(N))/\log\log N$ might tend to a constant
$c\in(1.74,1.78)$ as $N\to\infty$.

\begin{figure}[H]
\begin{center}
\includegraphics*[width=3.25in]{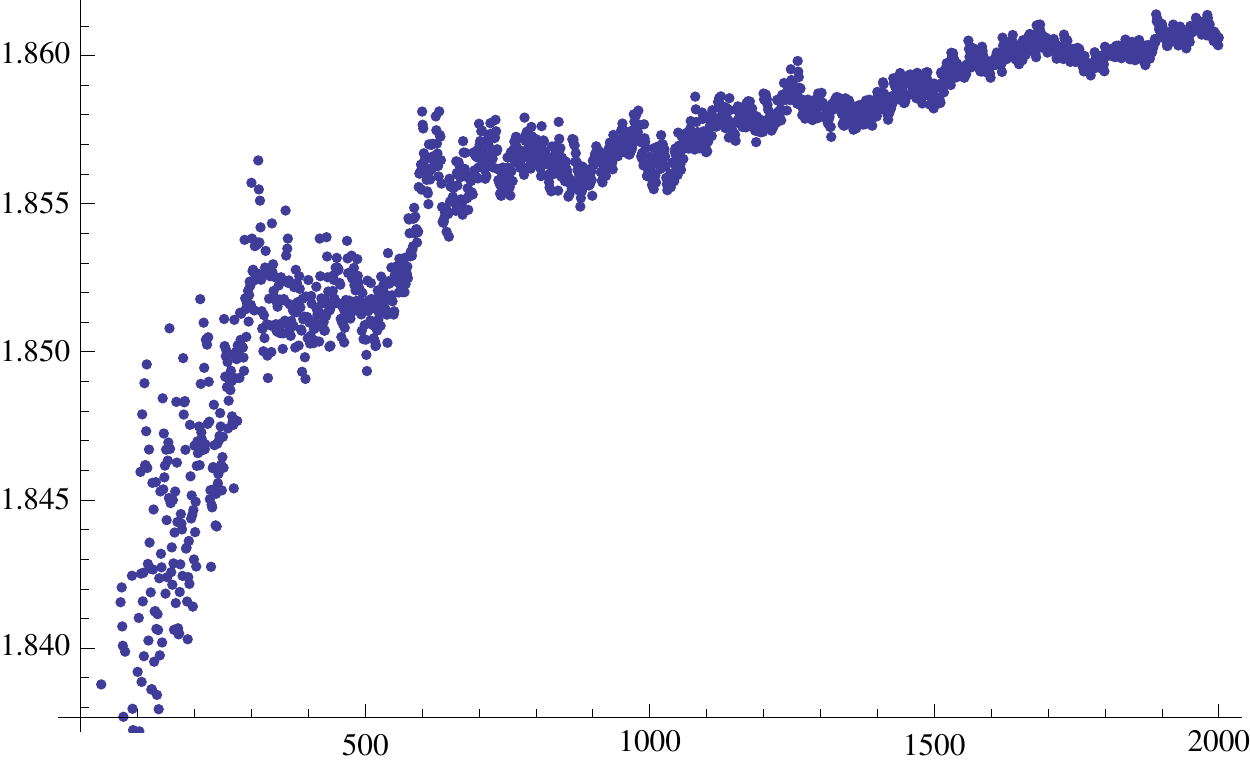}
\end{center}
\vspace*{-20pt}\caption{Values $E_2(N)/(\log N)^{1.74}$ with $2\le N\le 2000$}
\label{fig:1.7}
\end{figure}

\bigskip

\begin{figure}[H]
\begin{center}
\includegraphics*[width=3.25in]{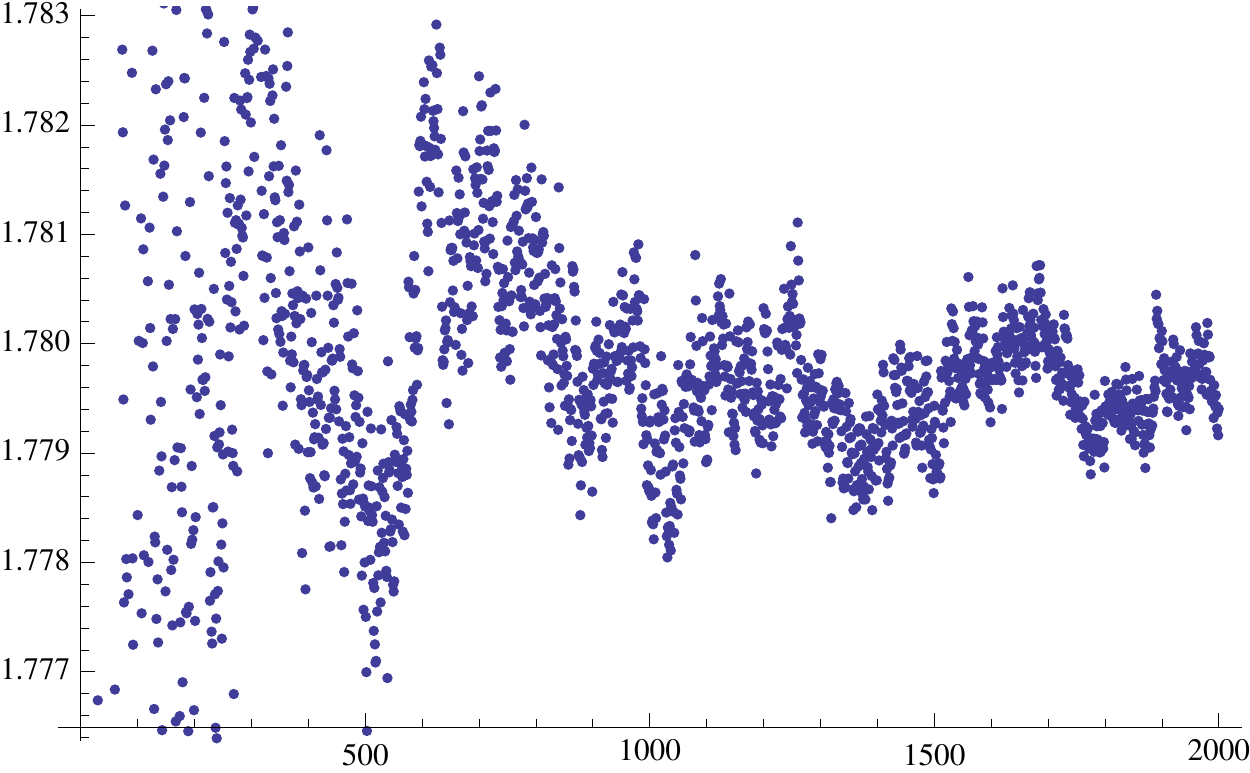}
\end{center}
\vspace*{-20pt}\caption{Values $E_2(N)/(\log N)^{1.762}$ with $2\le N\le 2000$}
\label{fig:1.76}
\end{figure}

\bigskip

\begin{figure}[H]
\begin{center}
\includegraphics*[width=3.25in]{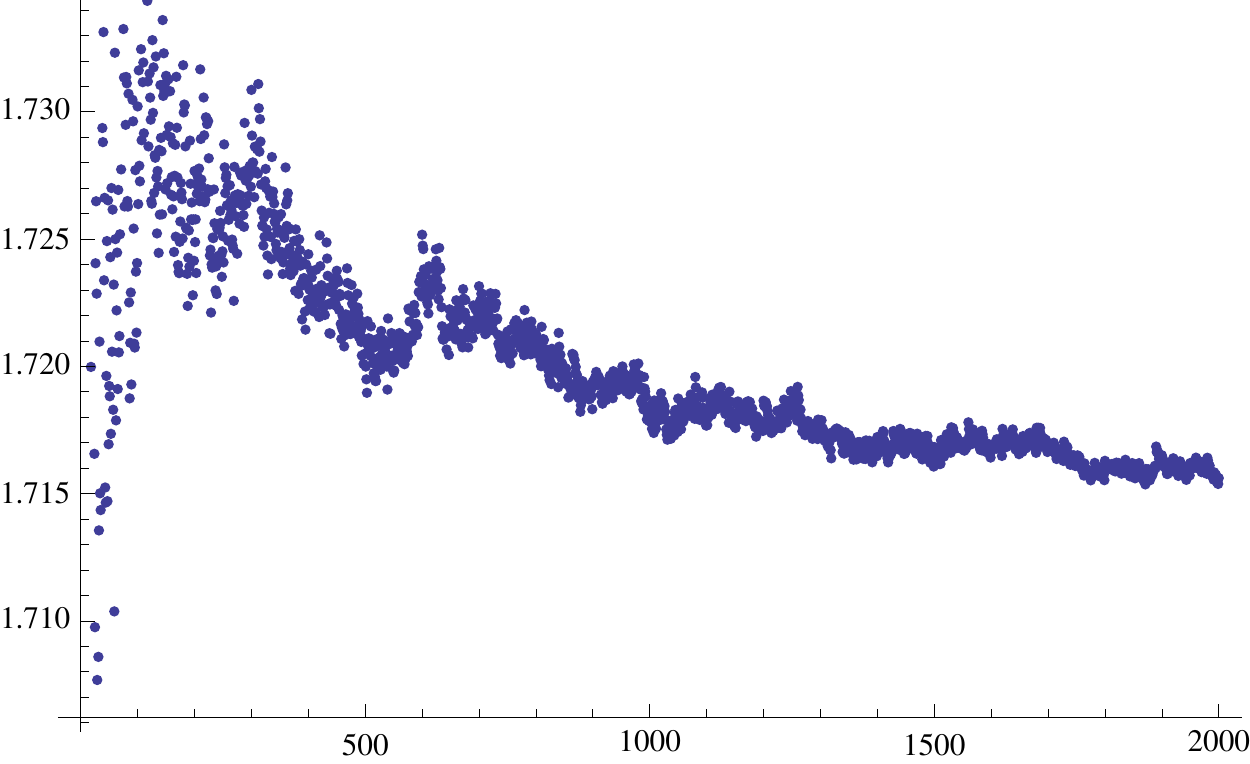}
\end{center}
\vspace*{-20pt}\caption{Values $E_2(N)/(\log N)^{1.78}$ with $2\le N\le 2000$}
\label{fig:1.8}
\end{figure}

To conclude this section, we provide Table~\ref{tab:modulus}
which, for any $n$ in the range $3\le n\le 40$,
give the modulus $q$ for which $A(n)=\Gnq/q^{1-1/n}$.

\begin{table}[H]
  \centering
\begin{tabular}{|c|c||c||c|c|}
\hline
$n$ & $q$ & & $n$ & $q$ \\
\hline
3 & 767484081 & & 22 & 1097192 \\
4 & 724880 & & 23 & 6533 \\
5 & 24816275 & & 24 & 11089264062240 \\
6 & 4606056 & & 25 & 1892365050125 \\
7 & 61103 & & 26 & 888749368 \\
8 & 35360 & & 27 & 122723007004143 \\
9 & 2302452243 & & 28 & 102143565680 \\
10 & 170568200 & & 29 & 59 \\
11 & 1541 & & 30 & 2221907019757425 \\
12 & 2343607353360 & & 31 & $\cdots\quad\cdots$ \\
13 & 4187 & & 32 & 2647898240 \\
14 & 488824 & & 33 & 26150655643931 \\
15 & 166568008135529 & & 34 & 14111 \\
16 & 6859840 & & 35 & 261183353167 \\
17 & 103 & & 36 & 766359604548720 \\
18 & 109951162776 & & 37 & 33227 \\
19 & $\cdots\quad\cdots$ & & 38 & 229 \\
20 & 75391144400 & & 39 & 728740376003003\\
21 & 2198500788029 & & 40 & 36338531600800 \\ \hline
\end{tabular}
\caption{Extreme moduli for $3\le n\le 40$}
\label{tab:modulus}
\end{table}

We remark that since $A(19)=A(31)=1$ we have
$A(19)=G_{19}(p^{18})$ for any prime $p\ne 19$
and $A(31)=G_{31}(p^{30})$ for any prime
$p\ne 31$; see the justification of~\eqref{eq:A1}.


\begin{thebibliography}{99}
 
\bibitem{CochPin}  T.~Cochrane and C.~Pinner, 
`Explicit bounds on monomial and binomial exponential sums,'
\emph{Quart.\ J.\ Math.} \textbf{62} (2011), 323--349. 

\bibitem{CohnElkies}
H.~Cohn and N.~Elkies,
`New upper bounds on sphere packings I,'
\emph{Ann.\ of Math.\ (2)} \textbf{157} (2003), no.~2, 689--714. 

\bibitem{ConSlo}
J.~H.~Conway and N.~J.~A.~Sloane,
\emph{Sphere packings, lattices and groups.} Third edition.
Grundlehren der Mathematischen Wissenschaften, \textbf{290}.
Springer-Verlag, New York, 1999.

\bibitem{HBK}
D.~R.~Heath-Brown and S.~V.~Konyagin,
`New bounds for Gauss sums derived from $k$th powers,
and for Heilbronn's exponential sum,'
\emph{Quart.\ J.\ Math.} \textbf{51} (2000), 221--235.

\bibitem{Hlawka}
E.~Hlawka, 
`Zur Geometrie der Zahlen,'
\emph{Math.\ Z.} \textbf{49} (1943), 285--312 (in German). 

\bibitem{KoSh}
S.~V.~Konyagin and  I.~E.~Shparlinski,
\emph{Character sums with exponential functions and their applications.}
Cambridge Tracts in Mathematics, \textbf{136}. Cambridge University Press, Cambridge, 1999.

\bibitem{LN} 
R.~Lidl and  H.~Niederreiter,
\emph{Finite fields.} Encyclopedia of Mathematics and its Applications, \textbf{20}.
Addison-Wesley Publishing Company, Advanced Book Program, Reading, MA, 1983.

\bibitem{NicRob}
J.-L.~Nicolas and G.~Robin,
`Majorations explicites pour le nombre de diviseurs de $N$,'
\emph{Canad.\ Math.\ Bull.} \textbf{26} (1983), no.~4, 485--492 (in French).

\bibitem{Shp} 
I.~E.~Shparlinski,
`On bounds of Gaussian sums,'
\emph{Matem.\ Zametki} \textbf{50} (1991), 122--130 (in Russian).

\bibitem{Ste}
S.~B.~Stechkin,
`An estimate for Gaussian sums',
\emph{Matem.\ Zametki} \textbf{17} (1975), 342--349 (in Russian).

\end{thebibliography}
\end{document}